\documentclass[a4paper,11pt]{article}

\usepackage{graphicx,psfrag}
\usepackage{amssymb,amsmath,amsthm}
\usepackage{enumerate}
\usepackage{a4wide}
\usepackage{hyperref}
\usepackage{array}

\allowdisplaybreaks[4]

\newtheorem{proposition}{Proposition}
\newtheorem{theorem}{Theorem}
\newtheorem{lemma}{Lemma}
\newtheorem{corollary}{Corollary}

\theoremstyle{definition}
\newtheorem{definition}{Definition}
\newtheorem{remark}{Remark}

\theoremstyle{definition}

\renewenvironment{proof}{\smallskip\noindent\emph{\textbf{Proof.}}%
  \hspace{1pt}}{\hspace{-5pt}{\nobreak\quad\nobreak\hfill\nobreak%
    $\square$\vspace{2pt}\par}\smallskip\goodbreak}

\newcommand{\C}[1]{\mathbf{C}^{#1}}

\renewcommand{\L}[1]{{\mathbf{L}^#1}}

\newcommand{\modulo}[1]{{\left|#1\right|}}

\newcommand{\reali}{{\mathbb{R}}}

\renewcommand{\epsilon}{\varepsilon}
\renewcommand{\phi}{\varphi}
\renewcommand{\theta}{\vartheta}

\newcommand{\Lip}{\mathinner\mathbf{Lip}}
\newcommand{\diam}{\mathop{\rm diam}}
\renewcommand{\d}[1]{\mathinner{\mathrm{d}{#1}}}
\renewcommand{\div}{\mathop{\rm div}}

\newcommand{\co}{\mathop{\rm co}}

\title{Optimal Control of Continuity Equations}
\date{October 18, 2014}

\begin{document}

\author{Nikolay~Pogodaev\footnotemark[1]}

\footnotetext[1]{Institute for System Dynamics and Control Theory of
Siberian Branch of Russian Academy of Sciences.}

\maketitle

\begin{abstract}
\noindent An optimal control problem for the continuity equation is considered. The aim of a ``controller'' is to maximize the total mass within a target set at a given time moment. The existence of optimal controls is established. For a particular case of the problem, where an initial distribution is absolutely continuous with smooth density and the target set has certain regularity properties, a necessary optimality condition is derived. It is shown that for the general problem one may construct a perturbed problem that satisfies all the assumptions of the necessary optimality condition, and any optimal control for the perturbed problem, is nearly optimal for the original one. 

  \medskip

  \noindent\textit{2000~Mathematics Subject Classification:} 49K20, 49J15

  \medskip

\end{abstract}

\maketitle

\section{Introduction}
\label{sec:introduction}
Consider a quantity distributed on $\reali^n$. Suppose that the distribution evolves along 
a controlled vector field $v = v(t,x,u)$ according to the law of mass conservation;
the control parameter $u$ can be chosen at every time $t$ from a compact set $U\subset \reali^m$. 
Given a time moment $T>0$ and a target set $A\subset\reali^n$, 
we aim at finding a control $u=u(t)$ that maximizes
the total mass within $A$ at the time moment $T$. 
More formally the problem can be written as follows
\begin{gather*}
\mbox{Maximize}\quad \int_{A}\rho(T,x)\d x\\
\mbox{subject to}\quad
\begin{cases}
\rho_t + \div{}\hspace{-2pt}_x \left(v\left(t,x,u(t)\right)\rho\right)=0,\\
\rho(0,x) = \rho_0(x),
\end{cases}
\quad u\in \mathcal{U},
\tag{$P_{\text{ac}}$}
\end{gather*}
where $\rho_0=\rho_0(x)$ is the density function of an initial distribution and
\begin{equation}
\label{eq:admcont}
\mathcal{U} = \left\{u(\cdot)\;\;\mbox{is measurable,}\;\; u(t)\in U\;\;\mbox{for all}\;\; t\in [0,T]\right\}
\end{equation}
is the set of admissible controls.

First, let us give several motivating examples.
\vspace{3pt}
\paragraph{Dynamical system with uncertain initial state.}

Let $x_0$ be an initial state of the following dynamical system
\begin{equation}
\label{eq:ode}
\dot x = v(t,x,u).
\end{equation}
Assume that one wants to find a control $u=u(t)$ that brings the state of the system to a target set $A$
at a given time $T$. 
Now let the precise initial state $x_0$ be unknown. Instead, assume that a \emph{probability distribution} of
$x_0$ on the state space is given.
In this case, one naturally looks for a control that maximizes the \emph{probability} of 
finding the state of the system within the target set at time $T$. This leads to problem $(P_{\text{ac}})$.

\vspace{3pt}
\paragraph{Flock control.}
Let $\rho_0$ characterize an initial distribution of sheep in a given area. 
Assume that the herd drifts along a vector field $v_s(x)$.
Assume, in addition, that there is a dog located at $u$. In this case,
a sheep located at $x$ obtains an additional velocity
\[
v_d(x,u) = \phi\left(|x-u|\right)(x-u).
\]
When $\phi$ is positive this means that the sheep tries to escape the dog.
If the interaction between the sheep are not relevant, the motion of the whole herd is described by
the equation
\[
\rho_t + \div{}\hspace{-2pt}_x\left(\left[v_s(x)+v_d(x,u)\right]\rho\right)=0.
\]
Typically, the dog wants to steer the herd to a target set $A$ at a given time. In this case an optimal 
strategy of the dog is determined by $(P_{\text{ac}})$.

\vspace{3pt}
\paragraph{Beam control.}

The motion of a charged particle in an electromagnetic field is described by the system
\begin{equation}
\label{eq:EMF}
\begin{cases}
\frac{d\mathbf{x}}{dt}  = \mathbf{v},\\
\frac{d(m\mathbf{v})}{dt}  = e\left(\mathbf{E}+\mathbf{v}\times\mathbf{B}\right) + \mathbf{G}(t,\mathbf{x},\mathbf{v}).
\end{cases}
\end{equation}
The particle is characterised by its charge $e$, rest mass $m_0$, and the relativistic mass $m = m_0/\sqrt{1-|\mathbf{v}|^2/c^2}$.
Above, $\mathbf{x} = (x_1,x_2,x_3)$ are the particle's coordinates, $\mathbf{v} = (v_1,v_2,v_3)$ are the velocities,
$\mathbf{E}$ and $\mathbf{B}$ are the electric and magnetic fields,
$c$ is the speed of light, and $\mathbf{G}(t,\mathbf{x},\mathbf{v})$ represents additional forces due to 
the interaction between the particle and the environment. 

Assume that the electromagnetic field depends on a parameter $u$, which can be chosen at every time moment $t$, i.e.,
\[
\mathbf{E} = \mathbf{E}(t,\mathbf{x},u),\qquad \mathbf{B} = \mathbf{B}(t,\mathbf{x},u).
\]
Then~\eqref{eq:EMF} can be rewritten in the form~\eqref{eq:ode} with 
$x = (x_1,x_2,x_3,v_1,v_2,v_3)$.

Producing a single particle (for example, in a particle accelerator) is extremely difficult. Instead,
a \emph{beam of particles} is produced. At the initial time moment every beam is characterized 
by its density function $\rho_0$ defined on the state space of~\eqref{eq:EMF}. 
Usually, one wants to focus the beams to ensure that the particles traveling in the accelerator collide.
Clearly, one can formulate this problem as $(P_{\text{ac}})$ with the target set
\[
A = \left\{(\mathbf{x},\mathbf{v})\;\colon\; |\mathbf{x}|\leq r,\;\; |\mathbf v|\leq c\right\},
\]
where $r$ is a desired radius of the beam.

The connection between dynamical systems with uncertain initial states and 
continuity equations is well-known. In the context of control theory it was mentioned 
in~\cite{BrockettLiouville2,Mazurenko,Propoi94}.
A basic mathematical model for flocks controlled by a leader is presented in~\cite{CGMP,ColomboMercierJNLS}. 
In contrast to these papers we neglect the interactions between flock's members. Formally, a nonlocal term describing 
the internal dynamics of the flock is omitted.
For models of the beam control we refer to~\cite{Ovsyannikov90}.

Existence results for optimal control problems governed by continuity equations seem to be missing in the current 
literature (at least for nonlinear vector fields $v$).
Necessary optimality conditions were
derived by D.\,A.~Ovsyannikov in~\cite{Ovsyannikov2006,Ovsyannikov80,Ovsyannikov90} and by A.\,I. Propo{\u\i} with his collaborators in~\cite{Propoi88,Propoi94,Propoi90}.
We remark that all the papers mentioned above consider the terminal cost functional
\[
\int_{\reali^n} \phi(x)\rho(T,x)\d x
\]
with \emph{smooth} $\phi$ and/or an integral cost functional. Moreover, the initial density $\rho_0$ is always assumed to be smooth.
Finally, let us mention the paper~\cite{Mazurenko} of S.\,S. Mazurenko, where
the dynamical programming method was developed,
and the papers~\cite{BrockettLiouville1,BrockettLiouville2} of R. Brockett, who
discussed controllability and some connections with stochastic equations.

In this paper we first study the existence of optimal controls. Next, assuming that the initial density $\rho_0$ is smooth
and the target set $A$ is sufficiently regular, we derive a necessary optimality condition.
Finally, we discuss the case of integrable $\rho_0$ and arbitrary $A$. More precisely, we replace $(P_{\text{ac}})$ 
by a perturbed problem $(P_{\epsilon})$ which satisfies all the assumptions of our necessary optimality condition. Then we show that 
every control that is optimal for $(P_{\epsilon})$ is ``nearly optimal'' for $(P_{\text{ac}})$.

We believe that our choice of the cost functional is not relevant
in the sense that the methods developed here should work in other cases as well. At the same time, in many situations 
it seems natural to maximize the total mass within a target set at a given time moment.
Moreover, our choice of the cost functional allows to derive a rather simple necessary optimality condition (see Section~\ref{sec:PMP}).

As we shall see later, dealing with the continuity equation for measures is more natural than dealing with that of for functions. 
For this reason, the main subject of our study is the following optimal control problem
\begin{gather*}
\mbox{Maximize}\quad \mu(T)(A)\\
\mbox{subject to}\quad
\begin{cases}
\mu_t + \div{}\hspace{-2pt}_x \left(v\left(t,x,u(t)\right)\mu\right)=0,\\
\mu(0) = \theta,
\end{cases}
\quad u\in \mathcal{U}.
\tag{$P$}
\end{gather*}
Now, $(P_{\text{ac}})$ is just a particular case of $(P)$, where the initial probability measure $\theta$ is absolutely continuous with
density $\rho_0$.

The paper is organized as follows. In Section~\ref{sec:prlim} we introduce basic notations, discuss the continuity equation
and generalized controls.  
In the next section we study the existence of optimal controls. 
Then, in Section~\ref{sec:PMP}, we prove a necessary optimality condition for a special case of $(P)$, where the initial distribution
is absolutely continuous with smooth density and the target set has certain regularity properties. 
In the last section we show that the general problem $(P)$ can be replaced by a perturbed problem $(P_{\epsilon})$ such that
$(P_{\epsilon})$ satisfies all the assumptions of our necessary optimality condition and every optimal control
for $(P_{\epsilon})$ is ``nearly optimal'' for $(P)$. For the reader's convenience, we place in Appendix a brief introduction to Young measures.

\section{Preliminaries}
\label{sec:prlim}

\subsection{Notation}

In what follows, $|x|$ is the Euclidean norm of $x\in \reali^n$ and $x\cdot y$ is the scalar product of 
$x,y\in\reali^n$. By $A_r$ and $A_r^\circ$ we mean the closed and the open 
$r$-neighbourhoods of $A\subset\reali^n$, i.e.,
\begin{align*}
A_r = \{x\;\colon\;|x-y|\leq r\mbox{  for some  } y\in A\},\\
A_r^\circ= \{x\;\colon\;|x-y|< r \mbox{  for some  } y\in A\}.
\end{align*}

Let $\mathcal{P}(\reali^n)$ denote the set of all probability measures on $\reali^n$.
We equip $\mathcal{P}(\reali^n)$ with the Prohorov distance:
\begin{multline*}
d_P(\theta_1,\theta_2) = 
\inf
\big\{\epsilon>0\;\colon\; 
\theta_1(A)\leq \theta_2(A_{\epsilon}^\circ)+\epsilon,\\
\theta_2(A)\leq \theta_1(A_{\epsilon}^\circ)+\epsilon\;\;
\mbox{for all Borel sets}\; A
\big\}.
\end{multline*}
The convergence
in the resulting metric space 
is exactly the narrow convergence of measures (see~\cite{BillingsleyConv,BogachevBook}).

Given a Borel map $f\colon\reali^n\to\reali^n$ and a probability measure $\theta$, define the \emph{pushforward 
(or image) measure} $\theta\circ f^{-1}$ by
\[
\left(\theta\circ f^{-1}\right)(A) = \theta\left(f^{-1}(A)\right)\quad\mbox{for all Borel sets }\;A\subseteq\reali^n.
\]
Recall the change of variables formula 
\[
\int \phi \d{\left(\theta\circ f^{-1}\right)} = \int \phi\circ f \d{\theta},
\]
which holds for all bounded Borel functions $\phi\colon\reali^n\to\reali$.

Below, $\lambda$ is the $n$-dimensional Lebesgue measure, $\sigma$ is the $(n-1)$-dimensional Hausdorff measure, $\rho\lambda$
is a probability measure which is absolutely continuous with respect to $\lambda$ and whose density is $\rho$.

\subsection{Flows of vector fields}

Consider a vector field $v\colon [0,T]\times\reali^n\to\reali^n$ satisfying the assumption
\[
\mbox{\textbf{(A0)}}\;\;
\begin{cases}
\mbox{The map $v=v(t,x)$ is measurable with respect to $t$; there are}\\ 
\mbox{positive constants $L$, $C$ such that, for all $t\in [0,T]$ and $x,x'\in\reali^n$,}\\
\modulo{v(t,x)-v(t,x')}\leq L|x-x'|
\quad\mbox{ and }\quad
\modulo{v(t,x)}\leq C\left(1+|x|\right).
\end{cases}
\]
Then there exists a unique solution $s\mapsto V_t^s(x)$ to the Cauchy problem
\[
\begin{cases}
\dot y(s) = v\left(s,y(s)\right),\\
y(t) = x.
\end{cases}
\]
The map $(s,t,x)\mapsto V_t^s(x)$ is called \emph{the flow of the vector field} $v$. 
Recall that the function $x\mapsto V_t^s(x)$ is a diffeomorphism and
\begin{gather*}
V_{\tau}^s\circ V_t^{\tau} = V_t^{s},\qquad V_s^s = \mathrm{Id},\\
\Lip(V_t^s)\leq e^{L|s-t|},\qquad \modulo{V_t^s(x)}\leq e^{C|s-t|}\left(1+|x|\right),
\end{gather*}
for any $s,t,\tau\in [0,T]$.
For details see, e.g.,~\cite{Betounes}.

\subsection{Continuity equation}

Let $\theta$ be a probability measure on $\reali^n$ and $v$ be a vector field satisfying Assumption \textbf{(A0)}. Consider the Cauchy problem
\begin{equation}
\label{eq:conteq}
\begin{cases}
\mu_t + \div{}\hspace{-2pt}_x \left(v\left(t,x\right)\mu\right)=0,\\
\mu(t_0)=\theta.
\end{cases}
\end{equation}

\begin{definition}
We say that a map $\mu\in \C0\left([t_0,T]; \mathcal{P}(\reali^n)\right)$ is a \emph{distributional solution} to~\eqref{eq:conteq}
if for any bounded Lipschitz continuous test function $\phi\colon [t_0,T]\times\reali^n\to \reali$, we have
\[
\int_{t_0}^T\int_{\reali^n}\left(\phi_t + v\cdot\nabla_{x}\phi \right)\d{\mu(t)}\d t = 0
\]
and $\mu(t_0)=\theta$.

\end{definition}

As is known~\cite{AmbrosioCrippa}, under Assumption \textbf{(A0)} there exists a unique distributional solution to~\eqref{eq:conteq}.
Moreover, this solution can be written in the form
\begin{equation}
\label{eq:representation}
\mu(t) = \theta\circ V_{t}^{t_0}\quad\mbox{for all }\; t\in [t_0,T],
\end{equation}
so that $\mu(t)$ is the pushforward of $\theta$ by $V_{t_0}^t$.

If $\theta$ is absolutely continuous, i.e., $\theta = \rho_0\lambda$, then~\eqref{eq:representation}
implies that $\mu(t)$ is absolutely continuous for each $t$. More precisely, we have $\mu(t)=\rho(t,\cdot)\lambda$,
where
\[
\rho(t,x) = \frac{\rho_0\circ V_{t}^{t_0}(x)}{\det DV_{t_0}^t(x)}.
\]

\subsection{Generalized controls}

Recall that by (usual) controls we mean the maps belonging to $\mathcal{U}$. 
We shall see later that in some cases it is more convenient to deal with \emph{generalized controls}.
We define such controls via Young measures. A brief discussion of Young measures
is given in Appendix; for a more solid introduction to the subject see, for instance,~\cite{ValadierCourse}.

\begin{definition}
A \emph{generalized control} is a Young measure $\nu\in \mathcal{Y}\left([0,T];\reali^m\right)$
whose support is contained in $[0,T]\times U$. 
The set of all generalized controls is denoted by $\mathcal{Y}\left([0,T];U\right)$.
\end{definition}

It is evident that a Young measure $\nu$ \emph{associated with} a usual control $u\in \mathcal{U}$
belongs to $\mathcal{Y}\left([0,T];U\right)$.

\begin{proposition}
\label{prop:compact}
The set of generalized controls $\mathcal{Y}\left([0,T];U\right)$ is compact 
in the space of Young measures $\mathcal{Y}\left([0,T];\reali^m\right)$.
\end{proposition}
\begin{proof}
The set $\mathcal{Y}\left([0,T];U\right)$ is relatively narrowly compact by Theorem~\ref{thm:prokhorov}. It remains to 
verify that it is closed. To this end, consider a converging sequence $(\nu_j)\subset \mathcal{Y}\left([0,T];U\right)$
with the narrow limit $\nu$. 
We only have to show that the support of $\nu$ is contained in $[0,T]\times U$. 
For each $j$, we have $\nu_j\left(\,]0,T[\times U^c\,\right)=0$, where $U^c$ is the complement of $U$. 
Since the set $]0,T[\times U^c$ is open, it follows from basic properties of the narrow convergence that
\[
0 = \liminf_{j\to\infty}\nu_{j}\left(\,]0,T[\times U^c\,\right) \geq \nu\left(\,]0,T[ \times U^c\,\right)=\nu\left([0,T]\times U^c\right).
\]
\end{proof}

\begin{definition}
We say that $t\mapsto \mu(t)$ is a \emph{trajectory corresponding to a generalized control} $\nu$ if $t\mapsto \mu(t)$ is a distributional solution to the Cauchy problem
\begin{equation}
\label{eq:cauchybar}
\begin{cases}
\mu_t + \div{}\hspace{-2pt}_x \left(\bar v\left(t,x\right)\mu\right)=0,\\
\mu(0) = \theta,
\end{cases}
\end{equation}
where
\begin{equation}
\label{eq:average_vf}
\bar v(t,x) = \int_{U} v(t,x,\omega)\d \nu_t(w)
\end{equation}
with $(\nu_t)_{t\in [0,T]}$ being the \emph{disintegration} of $\nu$.
\end{definition}

\subsection{Limit lemmas}

Here we gather several useful lemmas about sequences of pushforward measures and their limits.

\begin{lemma}
\label{lem:mu_j}
Suppose that $\theta_j\in \mathcal{P}(\reali^n)$ for $j=0,1,2\ldots$\,, and 
$f\colon\reali^n\to\reali^n$ is a continuous function.
If $\theta_j\to\theta_0$ narrowly, then $\theta_j\circ f^{-1}\to \theta_0\circ f^{-1}$ narrowly.
\end{lemma}
\begin{proof}
By the change of variables formula, we have
\[
\int\phi\d{(\theta_j\circ f^{-1})} = \int \phi\circ f\d\theta_j,
\qquad j=0,1,2,\ldots,
\]
for each bounded continuous test function $\phi$.
It remains to notice that $\phi\circ f$ is also bounded and continuous.
\end{proof}

\begin{lemma}
\label{lem:f_j}
Let $f_j\colon\reali^n\to\reali^n$, $j=0,1,2,\ldots$\,, be continuous functions, and
$\theta\in \mathcal{P}(\reali^n)$.
If $f_j\to f_0$ pointwise, then $\theta\circ f_j^{-1}\to \theta\circ f_0^{-1}$ narrowly.
\end{lemma}
\begin{proof}
For every bounded continuous test function $\phi$, we can write
\[
\int\phi\d{(\theta\circ f_j^{-1})} = \int \phi\circ f_j\d\theta,
\qquad j=0,1,2,\ldots
\]
Clearly, $\phi\circ f_j\to \phi\circ f_0$ pointwise. 
Now to complete the proof, we apply the Lebesgue dominated convergence theorem
to the right-hand side of the latter identity.
\end{proof}

\begin{lemma}
\label{lem:weakf}
Let
$v_{j}\colon [0,T]\times\reali^n\to\reali^n$, $j=0,1,2,\ldots$\,, be vector fields satisfying \textnormal{\textbf{(A0)}}
with common constants $L$ and $C$, and let
$x_j=x_j(t)$ denote a solution to the Cauchy problem
\[
\begin{cases}
\dot y(t) = v_j\left(t,y(t)\right),\\
y(0)=a,
\end{cases}
\quad j=0,1,2,\ldots\,.
\]
If $v_j(\cdot,x)\to v_0(\cdot,x)$ weakly in $\L1\left([0,T];\reali^n\right)$ for all $x$, then 
$x_j(t)\to x_0(t)$ for all $t\in [0,T]$.
\end{lemma}
\begin{proof}
For a fixed $t$ consider the obvious identity
\begin{align*}
x_j(t)-x_0(t) 
&= 
\int_0^t \left[v_j(s,x_j(s))-v_j(s,x_0(s))\right]\d s
\\
&+
\int_0^t \left[v_j(s,x_0(s))-v_0(s,x_0(s))\right]\d s.
\end{align*}
Observing that
\[
\modulo{v_j(s,x_j(s))-v_j(s,x_0(s))}\leq L\modulo{x_j(s)-x_0(s)}
\qquad\mbox{for all }\; s,
\]
we obtain, by Gronwall's inequality, the following estimate:
\[
\modulo{x_j(t)-x_0(t)}\leq |\alpha_j(t)|e^{Lt},
\]
where
\[
\alpha_j(t) = \int_0^t \left[v_j(s,x_0(s))-v_0(s,x_0(s))\right]\d s.
\]

To complete the proof it remains to show that $\lim_{j\to\infty}|\alpha_j(t)|=0$. For this purpose, take a sequence 
$(\tau_i)_{i=0}^N$ such that
\[
0 = \tau_0<\tau_1<\cdots<\tau_N=t\quad\mbox{and}\quad \tau_{i}-\tau_{i-1}=\frac{t}{N} \quad\mbox{for all }\;i,
\]
and consider the identity
\begin{align*}
\alpha_j(t) 
&=
\sum_{i=1}^{N}\int_{\tau_{i-1}}^{\tau_i}\left[v_j(s,x_0(s))-v_j(s,x_0(\tau_{i-1}))\right]\d s
\\
&+
\sum_{i=1}^{N}\int_{\tau_{i-1}}^{\tau_i}\left[v_j(s,x_0(\tau_{i-1}))-v_0(s,x_0(\tau_{i-1}))\right]\d s
\\
&+
\sum_{i=1}^{N}\int_{\tau_{i-1}}^{\tau_i}\left[v_0(s,x_0(\tau_{i-1}))-v_0(s,x_0(s))\right]\d s.
\end{align*}
By Assumption \textbf{(A0)}, we have
\[
\modulo{v_j(s,x_0(s))-v_j(s,x_0(\tau_{i-1}))}\leq L\modulo{x_0(s)-x_0(\tau_{i-1})}
\]
for all $i$, $j$, and $s$.
Moreover, by applying the standard ODE's technique, we can easily verify that
\[
\modulo{x_0(s) - x_0(\tau_{i-1})}\leq \frac{Mt}{N}\qquad\mbox{for all}\; s\in [\tau_{i-1},\tau_i],
\]
where $M>0$ depends only on $t$, $a$, and $C$. Therefore,
\[
|\alpha_j(t)|\leq \frac{2LMt^2}{N} + 
\sum_{i=1}^{N}\Big|\int_{\tau_{i-1}}^{\tau_i}\left[v_j(s,x_0(\tau_{i-1}))-v_0(s,x_0(\tau_{i-1}))\right]\d s\Big|.
\]
Passing to the limit as $j\to \infty$ and then as $N\to\infty$, we get $\lim_{j\to\infty}|\alpha_j(t)|=0$.

\end{proof}

\begin{lemma}
\label{lem:weakmu}
Let $\theta\in \mathcal{P}(\reali^n)$ and $\mu_j\colon[0,T]\to \mathcal{P}(\reali^n)$ 
denote a solution to the Cauchy problem
\[
\begin{cases}
\mu_t + \div{}\hspace{-2pt}_x\left(v_j(t,x)\mu\right)=0,\\
\mu_j(0)=\theta,
\end{cases}
\quad j = 0,1,2,\ldots,
\]
where the vector fields $v_j$ satisfy the assumptions of Lemma~\ref{lem:weakf}.
If $v_j(\cdot,x)\to v_0(\cdot,x)$ weakly in $\L1\left([0,T];\reali^n\right)$ for all $x$, then
\[
\mu_j(t)\to \mu_0(t)\;\;\mbox{narrowly\qquad for all  }\;\; t\in [0,T].
\]
\end{lemma}
\begin{proof}
Let $V_j$ denote the flow of $v_j$. In view of Lemma~\ref{lem:weakf}, the sequence
$(V_j)_{0}^t$ converges pointwise to $(V_0)_{0}^t$. Since $\mu_j(t) = \theta\circ (V_j)_{t}^{0}$,
the statement follows from Lemma~\ref{lem:f_j}.
\end{proof}

\begin{lemma}
\label{lem:convex}
Let $\phi\in \C0(\reali^m;\reali^n)$ and $\theta\in \mathcal{P}(\reali^m)$. 
If the support of $\theta$ is contained in a compact set $K\subset\reali^m$, then
\[
\int_K\phi(x)\d\theta(x)\in \co\phi(K),
\]
where $\co\phi(K)$ is the convex hull of $\phi(K)$.
\end{lemma}
\begin{proof}
Choose a sequence of probability measures of the form
\[
\theta_k = \sum_{i=1}^k c_i^k\delta_{x_i^k},\quad\mbox{with }\;x_i^k\in K\;\;\mbox{for all }\; i\mbox{ and }k,
\]
and such that $\theta_k\to\theta$ narrowly. This can be done as in~\cite[Example 8.1.6]{BogachevBook}.
Clearly,
\[
\int_K\phi(x)\d\theta_k(x) \to \int_K\phi(x)\d\theta(x).
\]
At the same time, we have
\[
\int_K\phi(x)\d\theta_k(x) = \sum_{i=1}^k c_i^k\phi(x_i^k)\in \co \phi(K).
\]
Since $\co \phi(K)$ is closed, the proof is complete.
\end{proof}

\section{Existence}
\label{sec:existence}

In this section we study the existence of optimal controls for $(P)$. 
Throughout the section, we make the following assumption
\[
\mbox{\textbf{(A1)}}\;\;
\begin{cases}
\mbox{The map $v\colon [0,T]\times \reali^n\times U\to\reali^n$ is continuous; there are positive}\\
\mbox{constants $L$, $C$ such that, for all $t\in [0,T]$, $u\in U$, and $x,x'\in\reali^n$,}\\
\modulo{v(t,x,u)-v(t,x',u)}\leq L|x-x'|
\quad\mbox{ and }\quad
\modulo{v(t,x,u)}\leq C\left(1+|x|\right).
\end{cases}
\]

First, we prove that $(P)$ has solutions within the space of generalized controls. Then, under additional assumptions,
we show that every trajectory corresponding to a generalized control can be produced by a usual one.

\begin{theorem}
\label{thm:exitence}
If the target set $A$ is closed, then $(P)$ has a solution within the space $\mathcal{Y}\left([0,T];U\right)$ of generalized controls.
\end{theorem}
\begin{proof}
Let $(\nu_j)$ be a maximizing sequence of generalized controls.
In view of Proposition~\ref{prop:compact}, we may assume, without loss of generality,
that $(\nu_{j})$ converges to some $\nu_0$. 

Consider the averaged vector fields $v_j$ defined by
\[
v_j(t,x) = \int_{U} v(t,x,\omega)\d{(\nu_{j})_t(\omega)},\qquad
j = 0,1,2,\ldots\,.
\]
It is easy to see that each $v_j$ satisfies Assumption \textbf{(A0)} 
with constants $L$ and $C$ as in Assumption \textbf{(A1)}. 
Hence, for each $j$, there exists a unique trajectory $\mu_j = \mu_j(t)$ corresponding 
to $\nu_j$.

For a given $x\in\reali^n$, we denote by $(t,u)\mapsto \bar v(t,x,u)$  a continuous extension of 
$(t,u)\mapsto v(t,x,u)$ to $[0,T]\times \reali^m$ with the property that
\[
\modulo{\bar v(t,x,u)}\leq C\left(1+|x|\right)\qquad\mbox{for all}\;\; t\in [0,T]\;\mbox{ and }\; u\in\reali^m.
\]
Such an extension exists by the Tietze theorem. It is clear that
\[
v_j(t,x) = \int_{\reali^m} \bar v(t,x,\omega)\d{(\nu_{j})_t(\omega)},\qquad
j = 0,1,2,\ldots\,.
\]

By using Proposition~\ref{prop:YMlimit}, we get
\[
v_j\left(\cdot,x\right)\to v_0(\cdot,x)\qquad\mbox{weakly in}\quad \L1\left([0,T];\reali^m\right)
\]
for every $x$. Now Lemma~\ref{lem:weakmu} asserts that
\[
\mu_{j}(T) \to \mu_0(T)\quad\mbox{narrowly}.
\] 
Therefore,
\[
\sup(P) = \limsup_{j\to\infty}\mu_{j}(T)(A)\leq \mu_0(T)(A),
\]
so that $\nu_0$ is optimal for $(P)$. This completes the proof.
\end{proof}

\begin{corollary}
\label{cor:existence}
Let $v$ have the form
\[
v(t,x,u) = v_0(t,x) + \sum_{i=1}^l \phi_i(t,u)v_i(t,x)
\]
for some real-valued functions $\phi_i$. Let the set
\[
\Phi(t,U) = 
\begin{pmatrix}
\phi_1(t,U)\\
\cdots\\
\phi_l(t,U)
\end{pmatrix}
\subset
\reali^l
\]
be convex, and the target set $A$ be closed. Then $(P)$ has a solution within the space $\mathcal{U}$ of
usual controls.
\end{corollary}
\begin{proof}
Let $\nu$ be an optimal generalized control. The corresponding averaged vector field is given by
\[
\bar v(t,x) = \int_{U} v(t,x,\omega)\d{\nu_t(\omega)} = v_0(t,x) + \sum_{i=1}^l \int_{U}\phi_i(t,\omega)\d{\nu_t(\omega)}v_i(t,x).
\]
Thus to complete the proof, it suffice to show that there exists a measurable function $u\colon[0,T]\to U$ such that
\[
\int_U\Phi(t,\omega)\d{\nu_t(\omega)} = 
\begin{pmatrix}
\int_U\phi_1(t,\omega)\d{\nu_t(\omega)}\\
\cdots\\
\int_U\phi_l(t,\omega)\d{\nu_t(\omega)}
\end{pmatrix}
=
\begin{pmatrix}
\phi_1\left(t,u(t)\right)\\
\cdots\\
\phi_l\left(t,u(t)\right)
\end{pmatrix}
=\Phi\left(t,u(t)\right)
\]
for almost every $t\in [0,T]$.
The latter follows from Filippov's lemma~\cite{Filippov62}, since
\[
\int_U\Phi(t,\omega)\d{\nu_t(\omega)} \in \Phi(t,U) \quad\mbox{for a.e.  }t\in [0,T],
\]
according to Lemma~\ref{lem:convex}.
\end{proof}

\section{Necessary optimality condition}
\label{sec:PMP}

\subsection{Statement}

Throughout this section, in addition to \textbf{(A1)}, we make the assumption\\
\begin{tabular}{c l}
\textbf{(A2)} &
The map $v=v(t,x,u)$ is twice continuously differentiable \\
& with respect to $x$.
\end{tabular}

\begin{theorem}
\label{thm:PMP}
Let $A$ be a compact set with the interior ball property
\footnote{i.e., $A$ is the union of closed balls of a certain positive radius $r$.}
and $\theta=\rho_0 \lambda$ with $\rho_0\in \C1(\mathbb{R}^n;\mathbb{R})$.
Let $\bar u$ be an optimal control for $(P)$ and $\bar\mu$ be the corresponding trajectory. 
Then for almost every $\tau\in [0,T]$, we have
\begin{multline}
\label{eq:optcon}
\int_{\partial A^{\tau}}
\bar\rho(\tau,x)\,
{v}\left(\tau,x,\bar u(\tau)\right)\cdot {n}_{A^{\tau}}(x)\d{\sigma(x)}\\
=\min_{\omega\in U}
 \int_{\partial A^{\tau}}
\bar\rho(\tau,x)\,
{v}(\tau,x,\omega)\cdot {n}_{A^{\tau}}(x)\d{\sigma(x)}.
\end{multline}
Here $A^{\tau} = \bar V_T^{\tau}(A)$ with $\bar V$ being the flow of the vector field $(t,x)\mapsto {v}\left(t,x,\bar u(t)\right)$,
${n}_{A^{\tau}}(x)$ is the measure theoretic outer unit normal to $A^{\tau}$ at $x$,
$\sigma$ is the $(n-1)$-dimensional Hausdorff measure, 
$\bar\rho(t,\cdot)$ is the density of $\bar\mu(t)$.
\end{theorem}

\begin{remark}
1. If $\partial A$ is an $(n-1)$-dimensional $\C2$ surface, then $A$ automatically has the interior ball property.
Moreover, each $\partial A^{\tau}$ is also an $(n-1)$-dimensional $\C2$ surface. 
Consequently, in this case $n_{A^{\tau}}(x)$ is the usual outer unit normal to $A^{\tau}$ 
at $x$, and $\sigma$ is the usual $(n-1)$-dimensional volume form.

2. The necessary optimality condition has a visual geometrical meaning. Let $\bar u=\bar u(t)$ be optimal.
Shift the target set $A$ along the vector field $(t,x)\mapsto v\left(t,x,\bar u(t)\right)$ backwards.
Denote the resulting image of $A$ at the time moment $\tau$ by $A^{\tau}$.
Then $\bar u(\tau)$ minimizes the outflow through $\partial A^{\tau}$ at almost every time moment $\tau$.
\end{remark}

The proof of Theorem~\ref{thm:PMP} is based on 
ideas of the Pontryagin Maximum Principle (see, e.g.,~\cite{BressanPiccoliBook}).
In addition, it relies on notions of the \emph{interior ball property} of a set and the \emph{directional derivative} 
of a real-valued function on $\mathcal{P}(\reali^n)$. We discuss these notions below, but first let us briefly outline the proof.

\paragraph{Sketch of the Proof:} Let $\bar u = \bar u(t)$ be an optimal control, $\bar\mu=\bar\mu(t)$ be the corresponding trajectory,
and $\bar V$ be the flow of $(t,x)\mapsto v\left(t,x,\bar u(t)\right)$.
Fix some $\tau\in [0,T]$ and consider the map $f(\theta) = \theta\circ \bar V^{\tau}_T(A)$ defined for all 
absolutely continuous $\theta$ with smooth density.
First, we show that $f$ is \emph{directionally differentiable} with respect to any vector field $w\in \C1(\reali^n;\reali^n)$ 
and compute its derivative $\partial_w f$. Then, we construct a needle variation $u_{\epsilon}$
of  $\bar u$ and show that
\[
\frac{d}{d\epsilon}\mu_{\epsilon}(T)(A)|_{\epsilon=0} = \partial_{\bar w}f\left(\bar\mu(\tau)\right)\quad\mbox{for almost all }\;\tau,
\]
where $\mu_{\epsilon}=\mu_{\epsilon}(t)$ is the trajectory corresponding to $u_{\epsilon}=u_{\epsilon}(t)$ and $\bar w$
is a certain vector field depending on $\tau$. Now, the necessary optimality condition is as follows:
\[
\partial_{\bar w}f\left(\bar\mu(\tau)\right)\leq 0\quad\mbox{for almost all }\;\tau.
\]


\subsection{Interior ball property}

Recall that a compact set $A$ has \emph{the interior ball property of radius $r>0$} if it is
the union of closed balls of radius $r$. Such sets have nice regularity properties;
in particular, $\partial A$ is $(n-1)$-rectifiable~\cite{Rataj} and
\begin{equation}
\label{eq:boundarybound}
\sigma(\partial A)\leq \frac{n \alpha_n (\diam A)^n}{2^nr},
\end{equation}
where $\alpha_n$ is the volume of the $n$-dimensional unit ball~\cite{AlvarezEtAl}. As a consequence, the measure theoretic
outer unit normal $n_A(x)$ to $A$ at $x$ exists for $\sigma$-a.e. $x\in \partial A$.

\begin{lemma}
\label{lem:preballs}
Let $\phi\colon\reali^n\to\reali^n$ be a diffeomorphism such that $\Lip(\phi^{-1})\leq b$. Then for any $A\subseteq \reali^n$,
we have
\[
\left(\phi(A)\right)_{\epsilon}\subseteq \phi(A_{b\epsilon})
\qquad\mbox{and}\qquad 
\left(\phi(A)\right)^\circ_{\epsilon}\subseteq \phi(A^\circ_{b\epsilon}).
\]
\end{lemma}
\begin{proof}
We prove only the first inclusion, for the second one is entirely similar.
Fix some $A\subseteq\reali^n$. Since $\phi$ is a diffeomorphism, it is enough to show that
\[
\phi^{-1}(A_{\epsilon})\subseteq \left[\phi^{-1}(A)\right]_{b\epsilon}.
\]
Indeed, for any $x\in A_\epsilon$, there exists $y\in A$ such that $|x-y|\leq \epsilon$. By the Lipschitz condition:
\[
|\phi^{-1}(x)-\phi^{-1}(y)|\leq b|x-y|\leq b \epsilon.
\]
Hence, for any $x\in A_\epsilon$, we have $\phi^{-1}(x)\in \left(\phi^{-1}(A)\right)_{b \epsilon}$.
\end{proof}

\begin{proposition}
\label{prop:mainbound}
Let $\theta$ be an absolutely continuous probability measure with a continuous density $\rho$.
Let $A\subset\reali^n$ be a compact set with the interior ball property of radius $r$. 
Let $\phi_1,\phi_2$ be diffeomorphisms such that
\[
\Lip(\phi_1),\quad \Lip(\phi_1^{-1}),\quad \Lip(\phi_2),\quad \Lip(\phi_2^{-1})
\]
are bounded from above by a positive constant $b$.
Then,
\[
\modulo{\theta\circ \phi_1(A)-\theta\circ \phi_2(A)}\leq  
M\,\frac{n \alpha_n (\diam A)^n}{2^{n-1} r}\,b^{n+1}\cdot \sup_{x\in A}\modulo{\phi_1(x)-\phi_2(x)},
\]
where $M = \max\{\rho(x)\;\colon\; x\in \phi_1(A)\,\triangle\, \phi_2(A)\}$.
\end{proposition}
\begin{proof}
First, notice that
\begin{align*}
\modulo{\theta\circ \phi_1(A)-\theta\circ \phi_2(A)}
\leq
\theta\left(\phi_1(A)\,\triangle\, \phi_2(A)\right)
\leq
M\lambda\left(\phi_1(A)\,\triangle\, \phi_2(A)\right)
\end{align*}
and
\[
\lambda\left(\phi_1(A)\,\triangle\, \phi_2(A)\right) = 
\lambda\left(\phi_1(A)\setminus \phi_2(A)\right)
+
\lambda\left(\phi_2(A)\setminus \phi_1(A)\right).
\]
Set $l = d_H\left(\phi_1(A),\phi_2(A)\right)$, where $d_H$ is the \emph{Hausdorff distance} between 
compact sets, i.e.,
\[
d_H(A,A') = \inf\left\{\epsilon>0\;\colon\;A\subseteq A'_{\epsilon}\;\;\mbox{and}\;\;A'\subseteq A_{\epsilon}\right\}.
\]
We get
\[
\phi_2(A)\subseteq [\phi_1(A)]_l
\quad\mbox{and}\quad \phi_1(A)\subseteq [\phi_2(A)]_l.
\]
Lemma~\ref{lem:preballs} yields
\[
[\phi_1(A)]_l\subseteq \phi_1(A_{lb}),
\]
so that
\[
\lambda\left(\phi_2(A)\setminus \phi_1(A)\right)
\leq 
\lambda\left(\phi_1(A_{lb})\setminus \phi_1(A)\right)
=
\lambda\left(\phi_1(A_{lb}\setminus A)\right).
\]
The Lipschitz condition for $\phi_1$ implies that
\[
\lambda\left(\phi_1(A_{lb}\setminus A)\right)\leq b^n\lambda(A_{lb}\setminus A),
\]
while the Reynolds transport theorem~\cite{LorenzReynolds} yields
\[
\lambda(A_{lb}\setminus A) = \lambda(A_{lb})- \lambda(A) = 
\int_0^{lb} \sigma(\partial A_t)\d t.
\]
Recalling~\eqref{eq:boundarybound}, we get
\[
\int_0^{lb} \sigma(\partial A_t)\d t
\leq \frac{n \alpha_n (\diam A)^n}{2^n}\int_0^{lb}\frac{\d t}{t+r}
\leq \frac{n \alpha_n (\diam A)^n}{2^n r}\,b\cdot l.
\]
Combining all the previous inequalities, we obtain the following estimate:
\[
\lambda\left(\phi_2(A)\setminus \phi_1(A)\right)
\leq 
\frac{n \alpha_n (\diam A)^n}{2^n r}\,b^{n+1}\cdot d_H(\phi_1(A),\phi_2(A)).
\]
The similar estimate holds for $\lambda\left(\phi_1(A)\setminus \phi_2(A)\right)$.
In order to complete the proof, it remains to observe that
\[
d_H\left(\phi_1(A),\phi_2(A)\right)\leq \sup_{x\in A}\modulo{\phi_1(x)-\phi_2(x)}.
\]
\end{proof}

\subsection{Directional derivative with respect to a vector field}


The basic idea employed here is to determine directions in $\mathcal{P}(\reali^n)$ by $\C1$ vector fields.
Hence the usual concept of directional derivative will be modified as follows.

\begin{definition}
\label{def:dirderiv}
We say that a map $f\colon \mathcal{P}(\reali^n)\to\reali$ is \emph{directionally differentiable with respect to a vector field} 
$w\in \C1(\reali^n;\reali^n)$ at a point $\theta\in \mathcal{P}(\reali^n)$ if there exists the limit
\[
\partial_w f(\theta) = \lim_{\epsilon\to 0+} \frac{f\left(\theta\circ W_{\epsilon}^0\right)-f(\theta)}{\epsilon},
\]
where $W$ is the flow of $w$.
In this case $\partial_w f(\theta)$ is called the \emph{directional derivative} of $f$ with respect to $w$ at $\theta$.
\end{definition}

\begin{proposition}
\label{prop:dirderive}
Let $A\subset\reali^n$ be a compact set with the interior ball property of radius $r$. 
The map $f\colon \mathcal{P}(\reali^n)\to\reali$ defined by
\begin{equation}
\label{eq:f}
f(\theta) = \theta\circ V_{T}^{\tau}(A),\qquad \theta\in \mathcal{P}(\reali^n),
\end{equation}
is directionally differentiable with respect to any vector field $w$ at every point $\rho\lambda\in \mathcal{P}(\reali^n)$,
with $\rho\in \C1(\reali^n;\reali)$, and its directional derivative is given by
\[
\partial_w f(\rho\lambda) = -\int_{V_T^{\tau}(\partial A)} 
\left(w(x) \cdot n_A(x)\right)\rho(x)
\d {\sigma(x)},
\]
where $n_A(x)$ is the measure theoretic outer unit normal to $A$ at $x$.
\end{proposition}
\begin{proof}
It follows from the Reynolds transport theorem that
\[
\partial_v f(\rho\lambda) = \frac{d}{d \epsilon}\int_{W^0_{\epsilon}\left(V_T^{\tau}(A)\right)}\rho(x)\d x\big|_{\epsilon=0}
= -\int_{V_T^{\tau}(A)} \div(\rho w)\d x.
\]
Now, by applying the Gauss-Green theorem, we complete the proof.
\end{proof}

\subsection{Proof of Theorem~\ref{thm:PMP}}

Let $\bar u$ be an optimal control and $\bar \mu$ be the corresponding trajectory. 
Given $\tau\in\;]0,T]$ and $\omega\in U$,
we define the perturbed control $u_{\epsilon}$ by the formula
\[
u_{\epsilon}(t) = 
\begin{cases}
\omega & \mbox{if}\;\; t\in [\tau- \epsilon, \tau],\\
\bar u(t) &\mbox{otherwise},
\end{cases}
\]
and denote the corresponding trajectory by $\mu_{\epsilon}$.

Let $\hat V$ and $\bar V$ be the flows of the vector fields 
$(t,x)\mapsto v\left(t,x,\omega\right)$ and $(t,x)\mapsto v\left(t,x,\bar u(t)\right)$.
Notice that
\begin{equation}
\label{eq:mueps}
\mu_{\epsilon}(\tau) = \bar\mu(\tau- \epsilon)\circ \hat V^{\tau- \epsilon}_{\tau}
= 
\bar\mu(\tau)\circ \bar V_{\tau- \epsilon}^{\tau}\circ \hat V^{\tau- \epsilon}_{\tau}.
\end{equation}
Finally, let $W$ denote the flow of the 
vector field
\begin{equation}
\label{eq:w}
w(x) = v(\tau,x,\omega) - v\left(\tau,x,\bar u(\tau)\right).
\end{equation}

\begin{lemma}
\label{lem:needle}
Under the assumptions of Theorem~\ref{thm:PMP}, we have
\[
\lim_{\epsilon\to 0+}\frac{f\left(\mu_{\epsilon}(\tau)\right) - f\left(\bar\mu(\tau)\circ W_{\epsilon}^0\right)}{\epsilon}=0
\qquad\mbox{for a.e.}\;\; \tau\in[0,T],
\]
where $f$ is defined by~\eqref{eq:f}.
\end{lemma}
\begin{proof}
\textbf{1.} In view of~\eqref{eq:mueps}, we have to verify that
\begin{equation}
\label{eq:toproof}
\hspace{-7pt}
\lim_{\epsilon\to 0+}
\frac{
	\bar\mu(\tau)\circ \bar V_{\tau- \epsilon}^{\tau}\circ \hat V^{\tau- \epsilon}_{\tau}(A)
	-
	\bar\mu(\tau)\circ W^0_{\epsilon}(A)
}{\epsilon}
=0
\quad\mbox{for a.e. }\; \tau\in[0,T].
\end{equation}

\textbf{2.} It follows from~\eqref{eq:w} that the map $W_{\epsilon}^0$ is Lipschitz continuous. 
Now, by Proposition~\ref{prop:mainbound} and by the Lipschitz property of $\bar V_{\tau- \epsilon}^{\tau}$, we obtain
\begin{align*}
\Big|
\bar\mu(\tau)\circ \bar V_{\tau- \epsilon}^{\tau}\circ \hat V^{\tau- \epsilon}_{\tau}(A)
-
\bar\mu(\tau)\circ W^0_{\epsilon}(A)
\Big|
\leq 
M
\sup_{x\in A}
\modulo{
\bar V_{\tau- \epsilon}^{\tau}\circ \hat V^{\tau- \epsilon}_{\tau}(x)
-
W^0_{\epsilon}(x)
}
\\
\leq
Me^{LT}
\sup_{x\in A}
\modulo{
\hat V^{\tau- \epsilon}_{\tau}(x)
-
\bar V^{\tau- \epsilon}_{\tau}\circ W^0_{\epsilon}(x)
}
\end{align*}
for some positive $M$.

\textbf{3.} For each $x\in \reali^n$, we have
\begin{align}
&W^0_{\epsilon}(x) - x
= 
-\int_{0}^{\epsilon} w\left(W^0_s(x)\right)\d s
\notag
\\
&\;=
-w\left(x\right) \epsilon
- \int_{0}^{\epsilon} \left[w\left(W_s^0(x)\right)-w\left(x\right)\right]\d s
\notag
\\
\label{eq:first}
&\;=
-w\left(x\right) \epsilon
- \int_{0}^{\epsilon}\int_0^1 Dw\left(\alpha W_s^0(x) +(1-\alpha)x\right)\d\alpha\cdot \left(W_s^0(x)-x\right)\d s.
\end{align}
On the other hand, it follows from \textbf{(A1)} that
\begin{equation}
\label{eq:firstestim}
\modulo{W_{s}^0(x) - x} \leq \int_0^{s} \modulo{w\left(W_{t}^0(x)\right)}\d t \leq\beta(x)\, \epsilon
\end{equation}
for all $s\in [0,\epsilon]$  and $x\in \reali^n$, where
\[
\beta(x) = 2C\left(1+e^{2CT}\left(1+|x|\right)\right).
\]

\textbf{4.} Fix some $x,y\in \reali^n$. Let $\chi_{\epsilon}=\chi_{\epsilon}(t)$ and $\bar\chi=\bar\chi(t)$ denote the solutions to the Cauchy problems
\[
\begin{cases}
\dot z = v\left(t,z,u_{\epsilon}(t)\right),\\
z(\tau) = x
\end{cases}
\qquad\mbox{and}\qquad
\begin{cases}
\dot z = v\left(t,z,\bar u(t)\right),\\
z(\tau) = y,
\end{cases}
\]
correspondingly. In particular, we have
\[
\hat V^{s}_{\tau}(x) = \chi_{\epsilon}(s),
\qquad
\bar V^{s}_{\tau}(y) = \bar\chi(s),
\quad\mbox{for all }\; s\in [\tau-\epsilon,\tau].
\]
Notice that
\begin{align*}
[\bar\chi(\tau &- \epsilon) - y] - \left[\chi_{\epsilon}(\tau - \epsilon) - x\right]
\\
&= 
\int_{\tau- \epsilon}^{\tau} 
\left[
v\left(s,\chi_{\epsilon}(s),\omega\right)
-
v\left(s,\bar\chi(s),\bar u(s)\right)
\right]
\d s
\\
&=
\int_{\tau- \epsilon}^{\tau} 
\left[
v\left(s,\chi_{\epsilon}(s),\omega\right)
-
v\left(s,\bar\chi(s),\omega\right)
\right]
\d s
\\
&\;\;\;+
\int_{\tau- \epsilon}^{\tau} 
\left[
v\left(s,\bar\chi(s),\omega\right)
-
v\left(s,\bar\chi(s),\bar u(s)\right)
\right]
\d s
\\ 
&=
\int_{\tau- \epsilon}^{\tau} 
\int_0^1 D_xv\left(s,\alpha\chi_{\epsilon}(s)+(1-\alpha)\bar\chi(s),\omega\right)\d\alpha\cdot
\left(\chi_{\epsilon}(s) - \bar\chi(s) \right)
\d s
\\
&\;\;\;+
\int_{\tau- \epsilon}^{\tau} 
\left[
v\left(s,\bar\chi(s),\omega\right)
-
v\left(s,\bar\chi(s),\bar u(s)\right)
\right]
\d s.
\end{align*}
In other words, for every $x,y\in \reali^n$, we get
\begin{multline}
\label{eq:second}
\bar V_{\tau}^{\tau- \epsilon}(y) - \hat V_{\tau}^{\tau- \epsilon}(x)
\\\hspace{-3pt}= y - x 
+
\int_{\tau- \epsilon}^{\tau} 
\left[
v\left(s,\bar V_{\tau}^{s}(x),\omega\right)
-
v\left(s,\bar V_{\tau}^{s}(x),\bar u(s)\right)
\right]
\d s
\\
\hspace{-6pt}+
\int_{\tau- \epsilon}^{\tau} 
\int_0^1 D_xv\left(s,\alpha\hat V_{\tau}^{s}(x)
+(1-\alpha)\bar V_{\tau}^{s}(y),\omega\right)
\d\alpha
\\
\cdot
\left(
\hat V_{\tau}^{s}(x) - \bar V_{\tau}^{s}(y)
\right)
\d s.
\end{multline}
Finally, it follows from \textbf{(A1)} that
\begin{align}
\hspace{-10pt}\modulo{\hat V^{s}_{\tau}(x) - \bar V^{s}_{\tau}(y)} 
&\leq |x-y|+
\int^{\tau}_s 
\left|
v\left(t,\hat V^{t}_{\tau}(x),\omega\right)
-
v\left(t,\bar V^{t}_{\tau}(y),\bar u(t)\right)
\right|
\d t
\notag
\\
\label{eq:secondestim}
&\leq 
|x-y| + \gamma(x,y)\, \epsilon
\end{align}
for all $s\in [\tau- \epsilon,\tau]$ and $x,y\in \reali^n$, where
\[
\gamma(x,y) = C\left(2 + e^{CT}\left(2+|x|+|y|\right)\right).
\]

\textbf{5.} 
Set $y = W_{\epsilon}^0(x)$. Now we combine~\eqref{eq:first}--\eqref{eq:secondestim}, and then recall~\eqref{eq:w} together with \textbf{(A1)} in order to estimate the derivatives of the vector fields. In this way we find that
\begin{multline*}
\modulo{
\hat V^{\tau- \epsilon}_{\tau}(x)
-
\bar V^{\tau- \epsilon}_{\tau}\circ W^0_{\epsilon}(x)
}\\
\leq
\int_{\tau- \epsilon}^{\tau} 
\left|
v\left(s,\bar V_{\tau}^{s}(x),\omega\right)
-
v\left(s,\bar V_{\tau}^{s}(x),\bar u(s)\right) - w(x)
\right|
\d s + N\epsilon^2
\end{multline*}
for every $x\in A$, where
\[
N = L\cdot \max_{x\in A}\left(3\beta(x) + \gamma\left(x,W_{\epsilon}^0(x)\right)\right).
\]
As a consequence, we get
\begin{multline*}
\modulo{
\bar\mu(\tau)\circ \bar V_{\tau- \epsilon}^{\tau}\circ \hat V^{\tau- \epsilon}_{\tau}(A)
-
\bar\mu(\tau)\circ W^0_{\epsilon}(A)
}
\\
\leq
Me^{LT}\cdot
\int_{\tau- \epsilon}^{\tau} 
\sup_{x\in A}\left|
v\left(s,\bar V_{\tau}^{s}(x),\omega\right)
-
v\left(s,\bar V_{\tau}^{s}(x),\bar u(s)\right) - w(x)
\right|
\d s \\+ Me^{LT}\, N \epsilon^2.
\end{multline*}
Notice that the map
\[
s\mapsto \sup_{x\in A}\left|
v\left(s,\bar V_{\tau}^{s}(x),\omega\right)
-
v\left(s,\bar V_{\tau}^{s}(x),\bar u(s)\right) - w\left(x\right)
\right|
\]
is integrable. Hence, applying the Lebesgue differentiation theorem and then looking at the definition of $w$, we conclude that~\eqref{eq:toproof}
holds for almost every $\tau\in [0,T]$. This completes the proof of Lemma~\ref{lem:needle}.
\end{proof}

Now in view of Proposition~\ref{prop:dirderive} and Lemma~\ref{lem:needle}, we see that
\begin{align*}
\lim_{\epsilon\to 0+}\frac{f\left(\mu_{\epsilon}(\tau)\right)-f\left(\bar\mu(\tau)\right)}{\epsilon}
&=
\lim_{\epsilon\to 0+}\frac{\mu_{\epsilon}(T)(A)-\bar\mu(T)(A)}{\epsilon}
\\
&=
-\int_{\bar V_T^{\tau}(\partial A)} 
\left(w(x) \cdot n_A(x)\right)\bar\rho(\tau,x)
\d {\sigma(x)}
\end{align*}
for almost every $\tau\in [0,T]$,
where $\bar\rho(\tau,\cdot)$ is the density of $\bar\mu(\tau)$. Since $\bar\mu$ is an optimal trajectory,
it follows that $\bar\mu(T)(A)\geq \mu_{\epsilon}(T)(A)$ for every $\epsilon$. Therefore,
\[
-\int_{\partial A^{\tau}} 
\left(w(x) \cdot n_A(x)\right)\bar\rho(\tau,x)
\d {\sigma(x)} \leq 0,
\]
where $A^{\tau} = \bar V_T^{\tau}(A)$.
By the definition of $w$, we obtain
\[
\int_{\partial A^{\tau}} 
\bar\rho(\tau,x)\,u\left(\tau,x,\bar u(\tau)\right) \cdot n_A(x)
\d {\sigma(x)}
\leq
\int_{\partial A^{\tau}} 
\bar\rho(\tau,x)\,u\left(\tau,x,\omega\right) \cdot n_A(x)
\d {\sigma(x)}
\]
for all $\omega\in U$ and almost every $\tau\in [0,T]$. The latter inequality is clearly equivalent to~\eqref{eq:optcon}.
The proof of Theorem~\ref{thm:PMP} is complete.

\section{Variational stability}
\label{sec:vs}

\subsection{Perturbation of the problem}

Our necessary optimality condition has two major drawbacks: the initial measure $\theta$ must be
absolutely continuous with smooth density and the target set $A$ must have the interior ball property.
In this section we show how to deal with control problems that do not satisfy the above assumptions.
In what follows, we suppose that $A$ is a compact set and $\theta$ is an arbitrary probability measure;
assumptions \textbf{(A1)} and \textbf{(A2)} are still valid.


Let $\eta = \eta(x)$ be the standard mollifier and $\eta_{\epsilon}(x) = \epsilon^{-n}\eta(x/\epsilon)$
for any positive $\epsilon$.
Define the \emph{convolution} of a probability measure $\theta$ and the map $\eta_{\epsilon}$ by
\[
\theta*\eta_{\epsilon}(x) = \int_{\reali^n}\eta_{\epsilon}(x-y)\d\theta(y).
\]
The function $\theta*\eta_{\epsilon}$ is known to be smooth~\cite{AmbrosioFuscoPallara}. By using the Fubini theorem,
we easily obtain the identity
\[
\int_{\reali^n}\theta*\eta_{\epsilon}(x)\d x = 1,
\]
so that $(\theta*\eta_{\epsilon})\lambda$ is a well defined probability measure.

Let us perturb the initial problem $(P)$ in the following way:
\begin{gather*}
\mbox{Maximize}\quad \mu(T)\left(A_{r \epsilon}\right)\\
\mbox{subject to}\quad
\begin{cases}
\mu_t + \div{}\hspace{-2pt}_x \left(v\left(t,x,u(t)\right)\mu\right)=0,\\
\mu(0) = (\theta*\eta_{\epsilon})\lambda,
\end{cases}
\quad u\in \mathcal{U},
\tag{$P_{\epsilon}$}
\end{gather*}
where $r = \max\left\{1,e^{LT}\right\}$.

\begin{theorem}
\label{thm:vs}
Let $\bar\nu_{\epsilon}$ be optimal for $(P_{\epsilon})$. 
Then every accumulation point $\bar\nu$ of the family $(\bar\nu_{\epsilon})_{\epsilon>0}$
is optimal for $(P)$ and $\max(P_{\epsilon})\to\max(P)$ as $\epsilon\to 0$.
Moreover, if $\theta$ is absolutely continuous and $\lambda(\partial A)=0$, then
\begin{equation}
\label{eq:nearlyopt}
\max(P) = \lim_{\epsilon\to 0}\mu(T,\theta,\bar\nu_{\epsilon})(A),
\end{equation}
where $t\mapsto\mu(t,\theta,\bar\nu_{\epsilon})$ denotes the trajectory corresponding to the 
initial distribution $\theta$ and the generalized control
$\bar\nu_{\epsilon}$.
\end{theorem}

\begin{remark}
1. Notice that $(P_{\epsilon})$ satisfies all the assumptions of Theorem~\ref{thm:PMP}.
Hence, instead of dealing directly with $(P)$, we may consider the perturbed problem $(P_{\epsilon})$ and try to find its solution with the aid of Theorem~\ref{thm:PMP}.
Suppose that we are managed to obtain such a solution. Let us denote it by $\bar\nu_{\epsilon}$.
Then Theorem~\ref{thm:vs} states that $\bar\nu_{\epsilon}$ approximates (in the narrow topology) a certain solution of the initial problem $(P)$. 
Moreover, if $\theta$ is absolutely continuous and $\lambda(\partial A)=0$, then $\bar\nu_{\epsilon}$ is nearly optimal for $(P)$ in the sense of equation~\eqref{eq:nearlyopt}.

2. In general one must perturb both the initial distribution $\theta$ and the target set $A$ even if
$A$ has already the interior ball property. Indeed, consider the following control system in $\reali^2$:
\[
\begin{cases}
\dot x = u,\\
x(0) = x_0,
\end{cases}
\quad |u|\leq 1.
\]
Assume that the target set $A$ is the unit ball centered at $0$ and $x_0 = (-2,0)$. 
Our aim is to bring the state of the system to $A$ at the time moment $T=1$. Clearly,
this can be done by means of the control $\bar u \equiv (1,0)$. Hence $\bar u$ is a 
solution to the maximization problem:
\begin{gather*}
\mbox{Maximize}\quad \mu(T)(A)\\
\mbox{subject to}\quad
\begin{cases}
\mu_t + \div{}\hspace{-2pt}_x \left(u\mu\right)=0,\\
\mu(0) = \delta_{x_0},
\end{cases}
\quad |u|\leq 1,
\tag{$P'$}
\end{gather*}
and $\max(P') = 1$. Perturbing only the initial distribution, we get the following problem:
\begin{gather*}
\mbox{Maximize}\quad \int_A\rho(T,x)\d x\\
\mbox{subject to}\quad
\begin{cases}
\rho_t + \div{}\hspace{-2pt}_x \left(u\rho\right)=0,\\
\rho(0,x) = \eta_{\epsilon}(x-x_0),
\end{cases}
\quad |u|\leq 1.
\tag{$P_{\epsilon}'$}
\end{gather*}
It is easy to see that $\max(P_{\epsilon}')$ tends to $\frac{1}{2}$, which is strictly less then $\max(P')$.
\end{remark}

\subsection{Preliminary lemmas}

\begin{lemma}
\label{lem:convolution}
Let $\theta\in \mathcal{P}(\reali^n)$ and $\theta_{\epsilon} = (\theta*\eta_{\epsilon})\lambda$.
Then, $d_P(\theta_{\epsilon},\theta)\leq \epsilon$.
\end{lemma}
\begin{proof}
By~\cite[page 72]{BillingsleyConv} it is enough to show that 
\[
\theta_{\epsilon}(A)\leq \theta(A^\circ_{\epsilon})+\epsilon
\]
for any Borel set $A$.
But this follows from the computations:
\begin{align*}
\int_A\int_{\reali^n}\rho_{\epsilon}(x-y)\d{\theta(y)}\d x 
&= 
\int_A\int_{A^\circ_{\epsilon}}\rho_{\epsilon}(x-y)\d{\theta(y)}\d x
\\
&=
\int_{A^\circ_{\epsilon}}\int_A\rho_{\epsilon}(x-y)\d x\d{\theta(y)}
\\
&\leq
\int_{A^\circ_{\epsilon}}\int_{\reali^n}\rho_{\epsilon}(x-y)\d x\d{\theta(y)}
\\
&=
\theta(A^\circ_{\epsilon}).
\end{align*}
\end{proof}

\begin{lemma}
\label{lem:convergence}
Let $\theta\in \mathcal{P}(\reali^n)$, $(\theta_{\epsilon})\subset \mathcal{P}(\reali^n)$, $\epsilon>0$, $r>0$, and $A$ be a closed set.
\begin{enumerate}[$(i)$]
\item If $d_P(\theta_{\epsilon},\theta)\leq r \epsilon$ for all $\epsilon$, then
$\lim_{\epsilon\to 0}\theta_{\epsilon}(A_{r \epsilon}) = \theta(A)$.
\item If $\theta_{\epsilon}\to\theta$ narrowly as $\epsilon\to 0$, then
$\liminf_{\epsilon\to 0}\theta_{\epsilon}(A_{r \epsilon})\leq\theta(A)$.
\end{enumerate}
\end{lemma}
\begin{proof}
The inequality $d_P(\theta_{\epsilon},\theta)\leq r\epsilon$ implies that
\[
\theta(A)\leq \theta_{\epsilon}(A_{r\epsilon}) + r\epsilon
\qquad\mbox{and}\qquad
\theta_{\epsilon}(A_{r\epsilon})\leq \theta(A_{2 r\epsilon}) + r\epsilon.
\]
Therefore,
\[
\limsup_{\epsilon\to 0}\theta_{\epsilon}(A_{r\epsilon})\leq \theta(A)\leq\liminf_{\epsilon\to 0}\theta_{\epsilon}(A_{r\epsilon}),
\]
and $(i)$ is established.

Since $\theta_{\epsilon}\to \theta$ narrowly as $\epsilon\to 0$, there exists a subsequence $(\theta_{\epsilon_j})$
with
\[
\epsilon_{j} < \frac{1}{j} 
\qquad\mbox{and}\qquad
d_P(\theta_{\epsilon_j},\theta)<\frac{1}{j}
\quad\mbox{for all }\;j.
\]
As a consequence, we have
\[
\theta_{\epsilon_j}(A_{r\epsilon_j})
\leq
\theta\left(A_{r \epsilon_j+1/j}\right) + \frac{1}{j}
\leq
\theta\left(A_{(r+1)/j}\right) + \frac{1}{j}.
\]
Passing to the limit as $j\to \infty$ yields
\[
\limsup_{j\to\infty}\theta_{\epsilon_j}(A_{r\epsilon_j})\leq \theta(A),
\]
which proves $(ii)$.
\end{proof}



\begin{lemma}
\label{lem:equicont}
Let $\theta_1$, $\theta_2$ be probability measures on $\reali^n$ and $\nu$
be a generalized control.
Then,
\[
d_P(\theta_1,\theta_2)\leq \epsilon
\quad\Rightarrow\quad
d_P\left(\mu(T,\theta_1,\nu),\mu(T,\theta_2,\nu)\right)\leq r \epsilon,
\]
where $t\mapsto\mu(t,\theta,\nu)$ is the trajectory corresponding to the 
initial distribution $\theta$ and the generalized control $\nu$
and $r=\max\{1,e^{LT}\}$.
\end{lemma}
\begin{proof}
Let $\phi = V^0_T$, where $V$ is the flow of the vector field
\[
(t,x)\mapsto \int_{\reali^m} v(t,x,\omega)\d\nu_t(\omega).
\]
Clearly, $\phi$ is a diffeomorphism, $\Lip(\phi^{-1})\leq r$, and
\[
\mu(T,\theta_1,\nu) = \theta_1\circ \phi,
\qquad
\mu(T,\theta_2,\nu) = \theta_2\circ \phi.
\]
The inequality $d_P(\theta_1,\theta_2)\leq \epsilon$ means that
\[
\begin{cases}
\theta_1\left(\phi(E)\right)\leq \theta_2\left(\left(\phi(E)\right)_{\epsilon+\delta}^\circ\right) + (\epsilon+\delta),\\
\theta_2\left(\phi(E)\right)\leq \theta_1\left(\left(\phi(E)\right)_{\epsilon+\delta}^\circ\right) + (\epsilon+\delta).
\end{cases}
\]
for any $\delta>0$ and any Borel set $E$.
It follows from Lemma~\ref{lem:preballs} that
\[
\left(\phi(E)\right)_{\epsilon+\delta}^\circ\subseteq \phi\left(E_{r(\epsilon+\delta)}^\circ\right).
\]
Therefore, for any $\delta>0$ and any Borel set $E$, we have
\[
\begin{cases}
\theta_1\circ \phi(E)\leq \theta_2\circ \phi\left(E_{r(\epsilon+\delta)}^\circ\right) + r(\epsilon+\delta),\\
\theta_2\circ \phi(E)\leq \theta_1\circ \phi\left(E_{r(\epsilon+\delta)}^\circ\right) + r(\epsilon+\delta).
\end{cases}
\]
That is, $d_P\left(\mu(T,\theta_1,\nu),\mu(T,\theta_2,\nu)\right)\leq r \epsilon$ as desired.
\end{proof}

\subsection{Proof of Theorem~\ref{thm:vs}}
\indent\textbf{1.} Let us show that
\begin{equation}
\label{eq:vs_first}
\nu_{\epsilon}\to\nu\;\;\mbox{narrowly}
\qquad\Rightarrow\qquad
\mu(T,\theta_{\epsilon},\nu_{\epsilon})\to \mu(T,\theta,\nu)\;\;\mbox{narrowly}.
\end{equation}
It follows from Lemma~\ref{lem:convolution} that $d_P(\theta_{\epsilon},\theta)\leq\epsilon$. Now
Lemma~\ref{lem:equicont} yields
\begin{equation}
\label{eq:reps}
d_P\left(\mu(T,\theta_{\epsilon},\nu),\mu(T,\theta,\nu)\right)\leq r \epsilon.
\end{equation}
Hence we can write
\begin{multline*}
d_{P}\left(\mu(T,\theta_{\epsilon},\nu_{\epsilon}),\mu(T,\theta,\nu)\right)
\leq
d_{P}\left(\mu(T,\theta_{\epsilon},\nu_{\epsilon}),\mu(T,\theta,\nu_{\epsilon})\right)\\
+
d_{P}\left(\mu(T,\theta,\nu_{\epsilon}),\mu(T,\theta,\nu)\right)
\leq 
r \epsilon +
d_{P}\left(\mu(T,\theta,\nu_{\epsilon}),\mu(T,\theta,\nu)\right).
\end{multline*}
Passing to the limit as $\epsilon\to 0$ and taking into account Lemma~\ref{lem:weakmu}, 
we obtain~\eqref{eq:vs_first}.

\textbf{2.} Since $\bar\nu_{\epsilon}$ is optimal for ($P_{\epsilon}$), we have
\begin{equation}
\label{eq:ws_onehalf}
\mu(T,\theta_{\epsilon},\nu)\left(A_{r \epsilon}\right)\leq \mu(T,\theta_{\epsilon},\bar\nu_{\epsilon})\left(A_{r \epsilon}\right)
\quad\mbox{for all }\;\nu.
\end{equation}
On the other hand, Lemma~\ref{lem:convergence}$(i)$, together with~\eqref{eq:reps}, implies that
\[
\lim_{\epsilon\to 0}\mu(T,\theta_{\epsilon},\nu)\left(A_{r \epsilon}\right)=\mu(T,\theta,\nu)(A) \qquad\mbox{for all  }\nu.
\]
Hence passing to the limit in~\eqref{eq:ws_onehalf} yields
\[
\mu(T,\theta,\nu)(A)\leq \liminf_{\epsilon\to 0}\mu(T,\theta_{\epsilon},\bar\nu_{\epsilon})\left(A_{r \epsilon}\right)\qquad\mbox{for all  }\nu,
\]
so that
\begin{equation}
\label{eq:vs_second}
\max(P)\leq \liminf_{\epsilon\to 0}\mu(T,\theta_{\epsilon},\bar\nu_{\epsilon})\left(A_{r \epsilon}\right)
=\liminf_{\epsilon\to 0}\left(\max(P_{\epsilon})\right).
\end{equation}

\textbf{3.} Assume that a subsequence $(\bar\nu_{\epsilon_j})$ converges narrowly to some $\bar \nu$.
Then, by~\eqref{eq:vs_first}, we have
\[
\mu(T,\theta_{\epsilon_j},\bar \nu_{\epsilon_j})\to \mu(T,\theta,\bar\nu)\quad\mbox{narrowly}.
\]
Now Lemma~\ref{lem:convergence}$(ii)$ implies that
\[
\liminf_{\epsilon\to 0}\mu(T,\theta_{\epsilon},\bar \nu_{\epsilon})
\left(A_{r\epsilon}\right)
\leq \mu(T,\theta,\bar\nu)(A). 
\]
Comparing this inequality with~\eqref{eq:vs_second}, we conclude that $\bar\nu$ is optimal for ($P$).

\textbf{4.} To complete the first part of the proof, it remains to show that
\[
\limsup_{\epsilon\to 0}\left(\max(P_{\epsilon})\right)\leq \max(P),
\]
or, equivalently,
\begin{equation}
\label{eq:vs_third}
\limsup_{\epsilon\to 0}\mu(T,\theta_{\epsilon},\bar\nu_{\epsilon})\left(A_{r \epsilon}\right)\leq \mu(T,\theta,\bar\nu)(A),
\end{equation}
where $\bar\nu$ is an accumulation point of $(\bar\nu_{\epsilon})$.
To this end, take any converging sequence $\left(\mu(T,\theta_{\epsilon_j},\bar\nu_{\epsilon_j})(A_{r \epsilon_j})\right)_j$ (we may always find one, since every $\mu(T,\theta_{\epsilon},\bar\nu_{\epsilon})(A_{r \epsilon})$ belongs to the compact segment $[0,1]$). 
Extract from it a subsequence $\left(\mu(T,\theta_{\epsilon_{j(k)}},\bar\nu_{\epsilon_{j(k)}})(A_{r \epsilon_{j(k)}})\right)_k$
such that $\bar\nu_{\epsilon_{j(k)}}$ converges narrowly to some $\bar\nu$. Then, by~\eqref{eq:vs_first}, we obtain
\[
\mu(T,\theta_{\epsilon_{j(k)}},\bar\nu_{\epsilon_{j(k)}}) \to\mu(T,\theta,\bar\nu).
\]
Now it follows from Lemma~\ref{lem:convergence}$(ii)$ that
\begin{multline}
\label{eq:vs_last}
\lim_{j\to\infty}\mu(T,\theta_{\epsilon_j},\bar\nu_{\epsilon_j})\left(A_{r \epsilon_{j}}\right)\\
=
\lim_{k\to\infty}\mu(T,\theta_{\epsilon_{j(k)}},\bar\nu_{\epsilon_{j(k)}})\left(A_{r \epsilon_{j(k)}}\right)
\leq\mu(T,\theta,\bar\nu)(A).
\end{multline}
Thus, every converging subsequence of 
$\left(\mu(T,\theta_{\epsilon},\bar\nu_{\epsilon})\left(A_{r \epsilon}\right)\right)$ satisfies~\eqref{eq:vs_last}.
This proves~\eqref{eq:vs_third}.

\textbf{5.} Let us prove the second part of the theorem.
Again, take any converging sequence $\left(\mu(T,\theta,\bar\nu_{\epsilon_j})(A)\right)_j$.
Extract a subsequence $\left(\mu(T,\theta,\bar\nu_{\epsilon_{j(k)}})(A)\right)_{k}$
such that $\bar \nu_{\epsilon_{j(k)}}$ converges narrowly to some $\bar\nu$ as $k\to \infty$.
In this case, by Lemma~\ref{lem:weakmu}, we have
\begin{equation}
\label{eq:2d_part1}
\mu(T,\theta,\bar\nu_{\epsilon_{j(k)}})\to \mu(T,\theta,\bar\nu).
\end{equation}
Since $\theta$ is absolutely continuous, we conclude that $\mu(T,\theta,\bar\nu)$ is absolutely continuous as well. Now the identity $\lambda(\partial A)=0$ imlies that $A$ is a \emph{continuty set} of $\mu(T,\theta,\bar\nu)$. Thus, it follows from~\eqref{eq:2d_part1} that
\begin{equation*}
\label{eq:2d_part2}
\lim_{j\to\infty}\mu(T,\theta,\bar\nu_{\epsilon_{j}})(A) = \lim_{k\to\infty}\mu(T,\theta,\bar\nu_{\epsilon_{j(k)}})(A) =  \mu(T,\theta,\bar\nu)(A) = \max(P).
\end{equation*}
In other words, every converging subsequence of 
$\left(\mu(T,\theta,\bar\nu_{\epsilon})\right)_{\epsilon}$ has $\max(P)$ as its limit.
This proves~\eqref{eq:nearlyopt} and completes the proof of Theorem~\ref{thm:vs}.

\appendix

\section{Young measures}

Below, $\lambda$ is the Lebesgue measure on $\reali^N$, $\Omega\subset \reali^N$ is a Borel set with $\lambda(\Omega)<\infty$, and
$\mathcal{B}(\Omega)$ is the $\sigma$-algebra of all Borel measurable subsets of $\Omega$.

\begin{definition}
A \emph{Young measure} is a positive Borel measure $\nu$ on $\Omega\times\reali^d$
such that $\nu(A\times\reali^d)=\lambda(A)$ for any Borel set $A\subseteq \Omega$.
The space of all Young measures is denoted by $\mathcal{Y}(\Omega;\reali^d)$.
\end{definition}

\begin{definition}
The \emph{narrow topology} on $\mathcal{Y}(\Omega;\reali^d)$ is the weakest topology
for which the maps 
$\nu\mapsto \int_{\Omega\times\reali^d}\phi\d\nu$
are continuous, where $\phi$ runs through the set of all bounded Charat\'eodory integrands
\footnote{A Charath\'eodory integrand is a $\mathcal{B}(\Omega)\otimes \mathcal{B}(\reali^d)$-measurable map 
which is continuous w.r.t. the second variable.} 
$\phi$ on $\Omega\times\reali^d$. 
\end{definition}

We always assume that $\mathcal{Y}(\Omega;\reali^d)$ is equipped with the narrow topology.

\begin{remark}
1. In the above definition one may replace Charath\'eodory 
integrands by continuous ones. This fact follows easily from the Scorza-Dragoni theorem. 
As a consequence, the narrow convergence of Young measures enjoys all the properties of the narrow convergence
of probability measures.


2. In the definition one may also take test functions of the form
\[
\phi(x,\xi) = 1_A(x)f(\xi),
\]
where $A$ is a Borel subset of $\Omega$ and $f$ is a bounded continuous function on $\reali^d$. 
Bearing that in mind, one may easily show that the narrow limit of a sequence of Young measures is a Young measure. Indeed,
taking $\phi = 1_A\cdot 1_{\reali^d}$, we get
\[
\lambda(A)=\lim_{j\to\infty} \nu_j(A\times\reali^d) = \nu(A\times\reali^d).
\]
\end{remark}

\begin{definition}
For a measurable function $u\colon\Omega\to\reali^d$ the \emph{associated Young measure} is the 
pushforward of $\lambda$ by the map $x\mapsto \left(x,u(x)\right)$.
\end{definition}

The set of Young measures associated with functions is dense in the set of all Young measures $\mathcal{Y}(\Omega;\reali^d)$.

Every Young measure $\nu$ can be described by its \emph{disintegration}, which is a family $(\nu_x)_{x\in\Omega}$
of probability measures on $\reali^d$ characterized by
\[
\nu(E) = \int_{\Omega} \nu_x(E_x)\d x,\qquad E\in \mathcal{B}(\Omega\times\reali^d),
\]
where $E_x=\{\xi\;\colon\;(x,\xi)\in E\}$. One may show that for any $\nu$-integrable $\psi$
\[
\int_{\Omega\times\reali^d}\psi\d\nu = \int_{\Omega}\Big[\int_{\reali^d}\psi(x,\xi)\d\nu_x(\xi)\Big]\d x.
\]
Consider, for instance, the Young measure $\nu$ associated with a measurable function $u\colon \Omega\to\reali^d$. Its disintegration is $\nu_x=\delta_{u(x)}$.

\begin{remark}
The notion of disintegration explains why we think of Young measures as generalized controls. Indeed, a map $u\colon\Omega\to\reali^d$ 
is a usual control: at every $x$ the value of the control parameter is prescribed and equal to $u(x)$. 
A Young measure $\nu$ is a generalized control: at every $x$ the control parameter is taken randomly according to 
the probability distribution $\nu_x$. Therefore Young measures are analogous to mixed strategies in Game Theory, 
where players choose their strategies randomly according to a probability distribution.
\end{remark}

\begin{proposition}
\label{prop:YMlimit}
Let $(\nu_j)$ be a sequence of Young measures. 
Let $\psi\colon\Omega\times\reali^d\to\reali^m$ be a bounded continuous map. 
If $(\nu_j)$ converges narrowly to a Young measure $\nu$, then the sequence of maps
$x\mapsto \int_{\reali^d}\psi\left(x,\xi\right)\d{(\nu_j)_x(\xi)}$ converges to the map $x\mapsto \int_{\reali^d}\psi(x,\xi)\d{\nu_x(\xi)}$
weakly in $\L1(\Omega;\reali^m)$.
\end{proposition}
\begin{proof}
Take a bounded measurable function $\phi\colon\Omega\to\reali^m$. The map $(x,\xi)\mapsto\phi(x)\cdot\psi(x,\xi)$
is a bounded Carath\'eodory integrand. Therefore,
\[
\lim_{j\to \infty}
\int_{\Omega\times\reali^d} \phi(x)\cdot\psi(x,\xi)\d\nu_j(x,\xi)= 
\int_{\Omega\times\reali^d} \phi(x)\cdot\psi(x,\xi)\d\nu(x,\xi).
\]
We may equivalently write
\[
\lim_{j\to\infty}\int_{\Omega}\phi(x)\cdot\int_{\reali^d}\psi\left(x,\xi\right)\d{(\nu_j)_x(\xi)}\d x = \int_{\Omega}\phi(x)\cdot\int_{\reali^d}\psi(x,\xi)\d{\nu_x(\xi)}\d x.
\]
Changing $\phi$ on a $\lambda$-negligible set does not affect both sides of the equality. Therefore, it holds for any 
$\phi\in\L\infty(\Omega;\reali^m)$.
\end{proof}

\begin{definition}
A subset $\mathcal{H}$ of $\mathcal{Y}(\Omega;\reali^d)$ is \emph{tight} if for every $\epsilon>0$ 
there exists a compact set $K\subset\reali^d$ such that
$\nu(\Omega\times K^c)< \epsilon$ for all $\nu\in \mathcal{H}$, where $K^c$ is the complement of $K$.
\end{definition}

\begin{theorem}[Prohorov]
\label{thm:prokhorov}
Every tight subset of $\mathcal{Y}(\Omega;\reali^d)$ is relatively narrowly compact.
\end{theorem}



\small{

  \bibliography{target}

\def\cprime{$'$} \def\cydot{\leavevmode\raise.4ex\hbox{.}}
\begin{thebibliography}{10}

\bibitem{AlvarezEtAl}
O.~Alvarez, P.~Cardaliaguet, and R.~Monneau.
\newblock Existence and uniqueness for dislocation dynamics with nonnegative
  velocity.
\newblock {\em Interfaces Free Bound.}, 7(4):415--434, 2005.

\bibitem{AmbrosioCrippa}
L.~Ambrosio and G.~Crippa.
\newblock Existence, uniqueness, stability and differentiability properties of
  the flow associated to weakly differentiable vector fields.
\newblock In {\em Transport equations and multi-{D} hyperbolic conservation
  laws}, volume~5 of {\em Lect. Notes Unione Mat. Ital.}, pages 3--57.
  Springer, Berlin, 2008.

\bibitem{AmbrosioFuscoPallara}
L.~Ambrosio, N.~Fusco, and D.~Pallara.
\newblock {\em Functions of bounded variation and free discontinuity problems}.
\newblock Oxford Mathematical Monographs. The Clarendon Press Oxford University
  Press, New York, 2000.

\bibitem{Betounes}
D.~Betounes.
\newblock {\em Differential equations: theory and applications}.
\newblock Springer, New York, second edition, 2010.

\bibitem{BillingsleyConv}
P.~Billingsley.
\newblock {\em Convergence of probability measures}.
\newblock Wiley Series in Probability and Statistics: Probability and
  Statistics. John Wiley \& Sons, Inc., New York, second edition, 1999.
\newblock A Wiley-Interscience Publication.

\bibitem{BogachevBook}
V.~I. Bogachev.
\newblock {\em Measure theory. {V}ol. {I}, {II}}.
\newblock Springer-Verlag, Berlin, 2007.

\bibitem{BressanPiccoliBook}
A.~Bressan and B.~Piccoli.
\newblock {\em Introduction to the mathematical theory of control}, volume~2 of
  {\em AIMS Series on Applied Mathematics}.
\newblock American Institute of Mathematical Sciences (AIMS), Springfield, MO,
  2007.

\bibitem{BrockettLiouville1}
R.~Brockett.
\newblock Optimal control of the {L}iouville equation.
\newblock In {\em Proceedings of the {I}nternational {C}onference on {C}omplex
  {G}eometry and {R}elated {F}ields}, volume~39 of {\em AMS/IP Stud. Adv.
  Math.}, pages 23--35. Amer. Math. Soc., Providence, RI, 2007.

\bibitem{BrockettLiouville2}
R.~Brockett.
\newblock Notes on the control of the {L}iouville equation.
\newblock In {\em Control of Partial Differential Equations}, Lecture Notes in
  Mathematics, pages 101--129. Springer Berlin Heidelberg, 2012.

\bibitem{CGMP}
R.~M. Colombo, M.~Garavello, M.~L{\'e}cureux-Mercier, and N.~Pogodaev.
\newblock Conservation laws in the modeling of moving crowds.
\newblock In {\em Hyperbolic Problems: Theory, Numerics, Applications},
  volume~8 of {\em AIMS Series on Applied Mathematics}, pages 467--474.
  American Institute of Mathematical Sciences, 2014.

\bibitem{ColomboMercierJNLS}
R.~M. Colombo and M.~L{\'e}cureux-Mercier.
\newblock An analytical framework to describe the interactions between
  individuals and a continuum.
\newblock {\em J. Nonlinear Sci.}, 22(1):39--61, 2012.

\bibitem{Filippov62}
A.~F. Filippov.
\newblock On certain questions in the theory of optimal control.
\newblock {\em J. SIAM Control Ser. A}, 1:76--84, 1962.

\bibitem{LorenzReynolds}
T.~Lorenz.
\newblock Reynold's transport theorem for differential inclusions.
\newblock {\em Set-Valued Anal.}, 14(3):209--247, 2006.

\bibitem{Mazurenko}
S.~S. Mazurenko.
\newblock The dynamic programming method in systems with states in the form of
  distributions.
\newblock {\em Vestnik Moskov. Univ. Ser. XV Vychisl. Mat. Kibernet.},
  (3):30--38, 2011.

\bibitem{Propoi88}
M.~V. Nikitin and A.~I. Propo{\u\i}.
\newblock Sootnoshenija dvojstvennosti i uslovija optimal'no­sti dlja
  nepreryvnyh potokovyh zadach.
\newblock {\em Proceedings of the All-Union Institute for System Research},
  (13):36--46, 1988.

\bibitem{Ovsyannikov2006}
D.~Ovsyannikov, A.~Ovsyannikov, M.~Vorogushin, Y.~Svistunov, and A.~Durkin.
\newblock Beam dynamics optimization: Models, methods and applications.
\newblock {\em Nuclear Instruments and Methods in Physics Research Section A:
  Accelerators, Spectrometers, Detectors and Associated Equipment}, 558(1):11
  -- 19, 2006.
\newblock Proceedings of the 8th International Computational Accelerator
  Physics Conference \{ICAP\} 2004 8th International Computational Accelerator
  Physics Conference.

\bibitem{Ovsyannikov80}
D.~A. Ovsyannikov.
\newblock {\em Matematicheskie metody upravleniya puchkami}.
\newblock Leningrad. Univ., Leningrad, 1980.

\bibitem{Ovsyannikov90}
D.~A. Ovsyannikov.
\newblock {\em Modelirovanie i optimizatsiya dinamiki puchkov zaryazhennykh
  chastits}.
\newblock Leningrad. Univ., Leningrad, 1990.

\bibitem{Propoi94}
A.~I. Propo{\u\i}.
\newblock Problems of the optimal control of mixed states.
\newblock {\em Avtomat. i Telemekh.}, (3):87--98, 1994.

\bibitem{Propoi90}
A.~I. Propo{\u\i} and A.~V. Pukhlikov.
\newblock Zadachi optimal'nogo upravlenija v sploshnyh sredah.
\newblock {\em Proceedings of the All-Union Institute for System Research},
  (7):60--77, 1990.

\bibitem{Rataj}
J.~Rataj and S.~Winter.
\newblock On volume and surface area of parallel sets.
\newblock {\em Indiana Univ. Math. J.}, 59(5):1661--1685, 2010.

\bibitem{ValadierCourse}
M.~Valadier.
\newblock A course on {Y}oung measures.
\newblock {\em Rend. Istit. Mat. Univ. Trieste}, 26(suppl.):349--394 (1995),
  1994.
\newblock Workshop on Measure Theory and Real Analysis (Italian) (Grado, 1993).

\end{thebibliography}

  \bibliographystyle{abbrv}

}

\end{document}